\documentclass[12pt,leqno]{article}

\usepackage{graphicx, fullpage}
\usepackage{amsmath,amssymb,amsthm,amscd, bm}
\usepackage{fancyhdr}
\usepackage[mathscr]{eucal}
\usepackage{amsfonts}
\usepackage[symbol]{footmisc}
\usepackage[T1]{fontenc}
\usepackage{xspace}
\usepackage{comment}
\begin{document}
\parskip=6pt

\theoremstyle{plain}
\newtheorem{prop}{Proposition}
\newtheorem{lem}[prop]{Lemma}
\newtheorem{thm}[prop]{Theorem}
\newtheorem{cor}[prop]{Corollary}
\newtheorem{defn}[prop]{Definition}
\theoremstyle{definition}
\newtheorem{example}[prop]{Example}
\theoremstyle{remark}
\newtheorem{remark}[prop]{Remark}
\numberwithin{prop}{section}
\numberwithin{equation}{section}

\newenvironment{rcases}
  {\left.\begin{aligned}}
  {\end{aligned}\right\rbrace}

\def\cal{\mathcal}
\newcommand{\cF}{\cal F}
\newcommand{\cG}{\cal G}
\newcommand{\cA}{\cal A}
\newcommand{\cB}{\cal B}
\newcommand{\cC}{\cal C}
\newcommand{\cI}{\cal I}
\newcommand{\cO}{{\cal O}}
\newcommand{\cE}{{\cal E}}
\newcommand{\cH}{{\cal H}}
\newcommand{\cU}{{\cal U}}
\newcommand{\cL}{{\cal L}}
\newcommand{\cM}{{\cal M}}
\newcommand{\cD}{{\cal D}}
\newcommand{\cK}{{\cal K}}
\newcommand{\cZ}{{\cal Z}}
\newcommand{\cQ}{{\cal Q}}

\newcommand{\fQ}{\frak{Q}}

\newcommand{\bC}{\mathbb C}
\newcommand{\bP}{\mathbb P}
\newcommand{\bN}{\mathbb N}
\newcommand{\bA}{\mathbb A}
\newcommand{\bR}{\mathbb R}
\newcommand{\oL}{\overline L}
\newcommand{\oP}{\overline P}
\newcommand{\op}{\overline \partial}
\newcommand{\oQ}{\overline Q}
\newcommand{\oR}{\overline R}
\newcommand{\oS}{\overline S}
\newcommand{\oc}{\overline c}
\newcommand{\bp}{\mathbb p}
\newcommand{\oD}{\overline D}
\newcommand{\oE}{\overline E}
\newcommand{\oC}{\overline C}
\newcommand{\of}{\overline f}
\newcommand{\ou}{\overline u}
\newcommand{\oU}{\overline U}
\newcommand{\ow}{\overline w}
\newcommand{\oy}{\overline y}
\newcommand{\oz}{\overline z}
\newcommand{\oxi}{\overline \xi}

\newcommand{\hg}{\hat G}
\newcommand{\hM}{\hat M}

\newcommand{\tpr}{\widetilde {\text{pr}}}
\newcommand{\tB}{\widetilde B}
\newcommand{\tx}{\widetilde x}
\newcommand{\ty}{\widetilde y}
\newcommand{\txi}{\widetilde \xi}
\newcommand{\teta}{\widetilde \eta}
\newcommand{\tna}{\widetilde \nabla}
\newcommand{\tth}{\widetilde \theta}

\newcommand{\diml}{\text{dim}}
\newcommand{\var}{\varepsilon}
\newcommand{\End}{\text{End }}
\newcommand{\loc}{\text{loc}}
\newcommand{\Symp}{\text{Symp}}
\newcommand{\Sympo}{\text{Symp}(\omega)}
\newcommand{\lam}{\lambda}
\newcommand{\na}{\nabla}
\newcommand{\Hom}{\text{Hom}}
\newcommand{\Ham}{\text{Ham}}
\newcommand{\Ker}{\text{Ker}}
\newcommand{\dist}{\text{dist}}
\newcommand{\psl}{\rm{PSL}}
\newcommand{\rk}{\text{rk}\,}

\newcommand{\psho}{\text{PSH}(\omega)}
\newcommand{\pshO}{\text{PSH}(\Omega)}

\renewcommand\qed{ }
\begin{titlepage}

\title{\bf On the Bergman kernels of holomorphic vector bundles}
\author{L\'aszl\'o Lempert \thanks{Research partially  supported by NSF grant DMS 1764167.
\newline 2020 Mathematics subject classification 32L05, 32Q99}\\
Department of  Mathematics\\
Purdue University\\West Lafayette, IN
47907-2067, USA}
\thispagestyle{empty}

\end{titlepage}
\date{}
\maketitle
\abstract
Consider a very ample line bundle $ E \to X$ over a compact complex manifold, endowed with a  hermitian metric of curvature $-i \omega $, and the space $\mathcal{O}(E)$ of its holomorphic sections. The Fubini--Study map associates with positive definite inner products $\langle \, , \rangle$ on $\mathcal{O}(E)$ functions FS$(\langle \, ,\rangle) \in \cal{H}_{\omega}=\{u \in C^{\infty}(X):\omega +i\partial\overline{\partial} u >0\}$. We prove that FS is an injective immersion, but its image in general is not closed in $\mathcal{H}_{\omega}$. To obtain a closed range, FS has to be extended to certain degenerate inner products. This we do by associating Bergman kernels with general inner products on the dual $\mathcal{O}(E)^*$, and the paper describes some simple properties of this association.

\endabstract

\section{Introduction}

There are several ways to associate Bergman kernels/sections/functions with a holomorphic vector bundle $E \to X$ over a compact base. Suppose $E \to X$ is a very ample line bundle with a  hermitian metric $h$, and we are given a positive definite inner product $\langle \, , \rangle$ on the space $\mathcal{O}(E)$ of its sections. The formula 
\begin{equation}\label{unouno}
\kappa(x)=\sup\Big\{\frac{h(s(x),s(x))}{\langle s,s\rangle}: s \in \mathcal{O}(E)\setminus \{0\}\Big\}, \qquad x \in X,
\end{equation}
defines an analog of Bergman's function. Bergman's original construction corresponds to $X$ not a compact manifold but an open subset of $\mathbb{C}^{n}$, $E=X  \times \mathbb{C} \to X$ the trivial line bundle, $h$ the trivial metric, and $\langle s,s\rangle= \int_X h(s,s)$, the integral with respect to Lebesgue measure. As there, in our case too, the sup in (\ref{unouno}) can be described in terms of an orthonormal basis of $(\mathcal{O}(E),\langle \, , \rangle)$.

Positive definite inner products form an open cone $\mathcal{H}_{E}$ in the finite dimensional real vector space of all inner products (by which we mean  hermitian symmectric sesquilinear maps 
$\mathcal{O}(E) \times \mathcal{O}(E) \to \mathbb{C}$). One version of the Fubini--Study map 
FS:$\mathcal{H}_{E} \to C^{\infty}(X)$ is defined by associating with $\langle\,\,  ,\, \rangle \in \mathcal{H}_E$  the logarithm of 
$\kappa$ in (\ref{unouno}). Here $C^{\infty} (X)$ stands for the space of smooth functions $X \to \mathbb{R}$. 
Yau in 1987 suggested to study a variant of this map when $E=K^{\otimes N}_X$ is a tensor power of the canonical bundle of $X$ and $N \to \infty$. Such asymptotic questions have received much attention in the greater generality when $E=L^{\otimes N}$ is a power of any positive line bundle, in the works of Catlin, Tian, Zelditch, and many others; a very recent contribution is by Wu, 
[Ca, T, W, Y, p. 139]. However, this paper is not about asymptotics, but properties of one fixed map FS:$\mathcal{H}_E \to C^{\infty}(X)$ and of related maps.
In fact, FS maps into the open subset 
$$
\mathcal{H}_{\omega} = \{u \in C^{\infty}(X): \omega + i\partial \tilde{\partial} u >0 \} \subset C^{\infty}(X),
$$ 
where $-i\omega$ is the curvature (1,1) form of $h$.

\begin{thm} If $E \to X$ is a very ample line bundle over a compact base, then FS$(\mathcal{H}_E) \subset \mathcal{H}_{\omega}$ 
is a not necessarily closed submanifold, and FS is a diffeomorphism between $\mathcal{H}_E$ and its image.\footnote{That (a 
different variant of) FS is injective is already stated in \cite[Theorem 1.1]{Hs}. Yet in email correspondence with the author we
realized that a relevant Lemma 3.1 of that paper is not quite correct, the norm $\|\,\|_{op}$ there needs to be defined with more care.}
\end{thm}

In Section 2 we will recall the relevant notions of infinite dimensional geometry. That 
FS$(\mathcal{H}_E) \subset \mathcal{H}_{\omega}$ may fail to be closed we will see in Example 6.2.

The question arises then, what is its closure? We give an answer in Section 6. The closure can be described in terms of
degenerations of positive definite inner products on $\mathcal{O}(E)$. The degenerations correspond to certain semidefinite inner products on the dual space $\mathcal{O}(E)^*$. It turns out that one can very naturally associate Bergman kernels and sections 
with {\sl any} inner product on $\mathcal{O}(E)^*$, for any holomorphic vector bundle $E\to X$; and this association has neat properties, some undoubtably familiar, that we describe in Sections 3, 4, 5.---In different problems 
in \cite{DW, W} Darvas and Wu  have already associated Bergman--type functions with certain inner
products/norms on $\cO(E)^*$, and studied the resulting `dual'  Fubini--Study maps.---Rather than sampling from the findings of
this paper, we introduce here one result that can be viewed as generalizing Theorem 1.1 to an arbitrary holomorphic vector bundle $E \to X$ over a compact base. We will be brief here, and give the details in Section 3.

Our $E$ has a conjugate bundle $\overline{E} \to X$, a complex vector bundle. As a real vector bundle it is the same as $E$, but multiplication by $\lambda \in \mathbb{C}$ in the fibers of $\overline{E}$ corresponds to multiplication by $\overline{\lambda}$ in $E$. Keep in mind that $\overline{E} \to X$ is not a holomorphic vector bundle. We denote the space of smooth sections of $E \otimes \overline{E}$ (tensor product over $\mathbb{C}$) by $C^{\infty}(E \otimes \overline{E})$. 

We continue to denote the cone of positive definite inner products on $\cO(E)$ by $\cH_E$. Any  
$\langle\, ,  \rangle \in \mathcal{H}_E$ determines a smooth section $k \in C^{\infty}(E \otimes \overline{E})$, its Bergman section. If $s_1, \dots,s_d$ form an orthonormal basis of $(\mathcal{O}(E), \langle\, ,  \rangle), $
\begin{equation}\label{unodos}
k(x)= \sum_{j=1}^d s_j(x) \otimes s_j(x), \qquad x \in X.
\end{equation}
Thus there is a map
\begin{equation}\label{unotres}
b:\mathcal{H}_E \to C^{\infty}(E \otimes \overline E)
\end{equation}
that associates with $\langle\, , \rangle \in \mathcal{H}_E$ the section $k$ in (\ref{unodos}). 

\begin{thm}For any holomorphic vector bundle over a compact base $b(\mathcal{H}_E) \subset C^{\infty}(E \otimes \overline{E})$ is a submanifold, and $b$ is a diffeomorphism between $\mathcal{H}_E$ and its image.
\end{thm}

\section{Basics of Fr\'echet manifolds} 
 
 Our reference to this material in Hamilton's paper \cite{Hm} (of which we will need very little). Suppose $F,G$ are Fr\'echet spaces over $\mathbb{R}$, and $U \subset F$ is open. A map $f: U \to G$ is of class $C^1$ if the directional derivatives 
 \begin{equation}\label{dosuno}
 df(x,\xi)= \lim_{t \to 0} \frac{f(x+t\xi)-f(x)}{t}
 \end{equation}
exist for all $x\in U$, $\xi \in F$, and define a continuous function $df: U \times F \to G$, called the differential of $f$. If $df$ is $C^1$, we say $f$ itself is $C^2$, and so on. The map $f$ is smooth if it is of class $C^k$ for all $k=1, 2,\dots$

Fr\'echet manifolds are Hausdorff topological spaces, glued together from open subsets of Fr\'echet spaces by smooth diffeomorphisms. One defines smoothness of maps between Fr\'echet manifolds by pulling back to open subsets of Fr\'echet spaces. The tangent bundle $TM$ of a Fr\'echet manifold $M$ is also a Fr\'echet manifold, in fact a Fr\'echet bundle over $M$. A smooth map $F:M \to N$ of Fr\'echet manifolds induces a smooth homorphism $TM \to TN$.

If $M$ is a Fr\'echet manifold, a subset $N \subset M$ is called a submanifold if the following holds: Every $x \in N$ has a neighborhood $U \subset M$, there are a Fr\'echet space $F$ and a closed subspace $G \subset F$, and there is a diffeomorphism $\Phi$ of $U$ on a neighborhood $V$ of $ 0 \in F$ such that $\Phi (U \cap N)=V \cap G$. The restrictions to $N$ of all such $\Phi$ define a manifold structure on $N$. Here our terminology deviates from Hamilton's. First, he requires that $N \subset M$ be a closed set, while we do not. Accordingly, we call $N \subset M$ a closed submanifold if it is a closed subset and a submanifold. Second, in his definition Hamilton stipulates that the subspace $G \subset F$ have a closed complement. If this also holds, we would call $N$ a direct submanifold; but the distinction will be of no consequence, as all submanifolds in this paper will be finite dimensional, hence automatically direct.

\section{Bergman Kernels} 

Fix a holomorphic vector bundle (of finite rank!) $E \to X$ over a compact base. In this section we will introduce two notions of Bergman kernels, the full Bergman kernel and its restriction to the diagonal that we call Bergman section.

If $Y$ is any complex manifold, we will write $\overline{Y}$ for the complex manifold with the same underlying differential manifold but opposite complex structure. Thus, a function $f$ on an open subset of $\overline{Y}$ is holomorphic if $\overline{f}$, viewed as a function on $Y$, is holomorphic. If $Y$ happens to be a complex vector space as well, then $\overline{Y}$ too acquires a complex vector space structure: As real vector spaces $\overline{Y}=Y$, but multiplication with $\lambda \in \mathbb{C}$ in $\overline{Y} $ is the same as multiplication with the complex conjugate $\overline{\lambda}$ in $Y$. Applying this construction in families, we see that if $\pi: E \to X$ is a holomorphic vector bundle, $\pi: \overline{E} \to \overline{X}$ becomes a holomorphic vector bundle. Keep in mind that $\mathcal{O}(E) = \mathcal{O}(\overline{E})$ as real vector spaces, but again, the complex structures are opposite. It follows that their duals are not equal, but there is a real linear isomorphism 
$\cO(E)^*\to\cO(\overline E)^*$
that associates with a linear form $\sigma$ on $\cO(E)$ its conjugate $\overline\sigma$, a linear form on $\cO(\overline E)$.
Similarly, if $x\in X$, there is a real linear isomorphism $E_x^*\to \overline E_x^*$ that sends a linear form 
$\xi\in E_x^*$ to its conjugate $\overline\xi\in\overline E_x^*$.

Occasionally we will view $\overline E$ as a bundle over $X$, not $\overline X$. When we do so, we will write 
$\overline E\to X$;  note that this is just a smooth complex vector bundle.

(Full) Bergman kernels of $E \to X$ are holomorphic sections of the external tensor product 
$E \boxtimes \overline E \to X \times \overline{X}$. Here 
\begin{equation}\nonumber
E \boxtimes \overline E=\coprod_{x \in X, y \in \overline{X}} E_x \otimes \overline{E}_y,  
\end{equation}
on which the obvious projection to $X \times \overline{X}$ and the local trivializations of $E$, $\overline{E}$ induce the structure of a holomorphic vector bundle. Commonly, Bergman kernels are associated with positive definite inner products on the space 
$\mathcal{O}(E)$ of sections. It will be more profitable, though, to associate them with inner products on the dual space 
$\mathcal{O}(E)^*$. Of course, positive definite (or nondegenerate) inner products on $\mathcal{O}(E)$ 
and $\mathcal{O}(E)^*$ are in bijective correspondence. But we will need to deal with degenerations, and they are easier to describe on $\mathcal{O}(E)^*$ than on $\mathcal{O}(E)$; they will be certain semidefinite inner products on $\mathcal{O}(E)^*$. 
At the same time it turns out that Bergman kernels can be associated with arbitrary inner products on $\mathcal{O}(E)^*$. 

We write $\mathcal{I}_E$ for the real vector space of inner products $\mathcal{O}(E)^*$. Elements 
$\langle \langle \, , \rangle \rangle \in \mathcal{I}_E $ are in bijective correspondence with pairs $(V, \langle \, , \rangle)$ of subspaces  $V \subset \mathcal{O}(E)$ and non-degenerate inner products $\langle \, , \rangle$ on $V$ as follows. A pair 
$(V,\langle \, ,\rangle)$ induces an inner product $\langle \, , \rangle'$ on the dual space $V^*$, a quotient of $\mathcal{O}(E)^*$. Pulling back $\langle \, , \rangle'$ by the quotient map of $ \mathcal{O}(E)^* \to V^*$ gives rise to the corresponding 
$\langle \langle \, , \rangle \rangle \in \mathcal{I}_E$ .
We write 
\begin{equation}\label{tresuno}
\langle \langle \, , \rangle \rangle= \delta(V, \langle \, , \rangle) \qquad \text{ and } \qquad \langle \langle \, , \rangle \rangle =\delta (\langle \, , \rangle) \quad\text{if } V=\mathcal{O}(E).
\end{equation}
When $V=\mathcal{O}(E),$  $\delta( \langle \, , \rangle)$ is just the dual inner product on $\mathcal{O}(E)^*$. Conversely, given $\langle \langle \, , \rangle \rangle \in \mathcal{I}_E$, let 
$$
N=\{\sigma \in \mathcal{O}(E)^*: \langle \langle \sigma, \cdot \rangle\rangle=0\}
$$
denote its kernel, and $V \subset \mathcal{O}(E)$ the annihilator of $N$,
$$ 
V=\{s \in \mathcal{O}(E): \sigma(s)=0 \text{ for all } \sigma \in N\}.
$$
The inner product $\langle \langle \, , \rangle \rangle$ descends to a nondegenerate inner product $\langle \, , \rangle '$ on 
$\mathcal{O}(E)^* / N$. Any $s\in V$ induces a linear form $\sigma \mapsto \sigma(s)$ on $\mathcal{O}(E)^*$, that factors through 
$\mathcal{O}(E)^* \to \mathcal{O}(E)^* / N$. Hence it can be represented as 
$\sigma \mapsto \langle \langle \sigma, \tau \rangle \rangle$with  some $ \tau= \tau_s \in \mathcal{O}(E)^*$ determined modulo 
$N$. The nondegenerate inner product on $V$ given by  $\langle s, t \rangle = \langle \langle \tau_t, \tau_s\rangle\rangle$ then 
satisfies $\delta( V, \langle \, , \rangle)=\langle\langle \, ,\rangle \rangle$.

We are ready to define the Bergman kernel of $\langle \langle \, , \rangle \rangle \in \mathcal{I}_E$. If ev$_x$ denotes the evaluation map 
$$
\textnormal{ev}_x: \mathcal{O}(E) \ni s \mapsto s(x) \in E_x,\qquad x \in X,
$$
the Bergman kernel in question is a section $K$ of $E \boxtimes \overline{E}$ with the property that 
\begin{equation}\label{tresdos}
\langle\langle \xi \textnormal{ev}_x, \eta \,\textnormal{ev}_y \rangle\rangle=(\xi \otimes \overline{\eta}) K(x,y) \qquad \textnormal{for all } x,y \in X, \, \xi \in E_{x}^*, \,\eta \in E_y^*.
\end{equation}
One way to see that such $K(x,y)$ exists is to note first that elements $\sum a_j \otimes b_j$ of a tensor product 
$A \otimes  B$ determine a linear form $\varphi$ on $A^*\otimes B^*$,
$$ 
\varphi \left(\sum \alpha_i \otimes \beta_i \right)= \sum_{i, j} \alpha_i(a_j)\beta_i(b_j), \qquad \alpha_i \in A^*, \beta_i \in B^*,
$$
and when $A,B$ are finite dimensional, all linear forms $\varphi$ on $A^*\otimes B^*$ arise in this way; second, that for 
fixed $x,y$, the left hand side of \eqref{tresdos} determines a linear form on $E_x^* \otimes \overline{E}^*_y$. Clearly, $K(x,y)$ is unique.
 
 Since the left hand side of \eqref{tresdos} depends holomorphically on $\xi\in E^*,\overline\eta\in\overline E^*$, it follows that 
 $K$ is a holomorphic section of $E \boxtimes \overline{E}$. If $\langle \langle \, , \rangle \rangle =\delta (V, \langle \, ,\rangle)$, then $K$ can be represented in terms of a signed orthonormal basis of $(V, \langle \, , \rangle)$. Suppose a basis
 $s_1,\dots,s_{p+q} \in V$ satisfies
 \[ 
 \langle s_i, s_j \rangle= \begin{cases}
 					\delta_{ij} & \textnormal{ if } i \leq p \\
					-\delta_{ij} & \textnormal{ if } i >  p  
					\end{cases}
  \quad(\textnormal{Kronecker symbol}). 
  \] 
 Then
 \begin{equation}\label{trestres}
 K(x,y)= \sum_{j=1}^p s_j (x) \otimes s_j(y) - \sum_{j=p+1}^{p+q} s_j(x) \otimes s_j(y).
 \end{equation}
  
 Indeed, choose $\sigma_1,\dots, \sigma_{p+q} \in \mathcal{O}(E)^*$ so that their classes $[\sigma_j] \in V^*$ form a dual basis; and complete them by elements of the kernel $N$ of $\langle\langle\,, \rangle\rangle$
 to a basis $ \sigma_1,\dots, \sigma_r$ of $\mathcal{O}(E)^*$. Write 
 $$
 \xi \textnormal{ev}_x=\sum_1^r \alpha_i \sigma_i, \quad \eta \,\textnormal{ev}_y= \sum_1^r \beta_i \sigma_i, \qquad \alpha_i, \beta_i \in \mathbb{C}.
 $$
We have $\xi s_j(x)=\xi \textnormal{ev}_x s_j=\alpha_j$ for $j=1,\dots,p+q$. Hence
$$
\langle \langle  \xi \textnormal{ev}_x, \eta \,\textnormal{ev}_y \rangle \rangle=\langle \langle \sum_1^r \alpha_j \sigma_j, \sum_1^r \beta_j \sigma_j \rangle \rangle = \sum_1^p \alpha_j \overline{\beta}_j - \sum_{p+1}^{p+q} \alpha_j \overline{\beta}_j=(\xi \otimes \overline{\eta}) K(x,y)
$$
if $K$ is given by \eqref{trestres}.

The familiar reproducing property of $K$ carries over to the generality we consider here. If $\eta\in E_y^*$, write
$\overline\eta K$ for $(\text{id}_E\otimes\overline\eta) K(\cdot,y)\in V$.

\begin{lem}
If $s\in V$ and $\eta\in E_y^*$, then 
\begin{equation} 
\eta s(y)=\langle s,\overline\eta K\rangle.
\end{equation}
\end{lem}
\begin{proof}
With $s_j$ as in (3.3) and $a_j\in\bC$ let $s=\sum_j a_j s_j$. The right hand side of (3.4) is
$$
\sum_{j=1}^{p+q}a_j\Big\langle s_j,\sum_{i=1}^p s_i\overline\eta s_i(y)-\sum_{i=p+1}^{p+q}s_i\overline\eta s_i(y)\Big\rangle=
\sum_{j=1}^{p+q} a_j\eta s_j(y)= \eta s(y),
$$
as claimed.
\end{proof}

The restriction of $K$ to the diagonal
\begin{equation}\label{trescuatro}
k(x)= K(x,x)=\sum_1^p s_j(x) \otimes s_j(x) -\sum_{p+1}^{p+q} s_j(x) \otimes s_j(x) \in E_x \otimes \overline{E}_x
\end{equation}
is also of interest. It is a smooth section of $E \otimes \overline{E} \to X$ that we call the Bergman section. By pairing with
$E^*\otimes \overline E^*$ it induces a 
not necessarily positive definite hermitian metric on this latter. If this induced metric is nondegenerate, by duality it further 
induces a nondegenerate, but again not necessarily definite hermitian metric on $E\otimes\overline E$. One version of the 
Fubini--Study map, different from the map introduced in Section 1,  
takes a positive definite inner product $\langle\,, \rangle\in\cH_E$ on $\cO(E)$, 
the dual inner product $\langle\langle\,, \rangle\rangle$ on $\cO(E)^*$, and its Bergman section $k$. Often it is convenient to
define the Fubini--Study map to send $\langle\,, \rangle$ to the hermitian metric of $E\otimes\overline E$ induced by $k$, necessarily
positive definite this time. This has the advantage over the notion defined in the Introduction, through (1.1), that it is independent of 
the choice of any hermitian metric $h$ on $E$, and works not only for line bundles. Nonetheless, in this paper we will stick 
with the map FS as defined in the introduction, more natural for our purposes. (5.1), (5.8) will clarify how the two variants of
the Fubini--Study map (i.e. the two variants of the function $\kappa$ in (1.1) and (5.1)) are related.

\section{Basic properties} 

Both the Bergman kernel $K$ and the Bergman section $k$ have an obvious symmetry property. There is an involutive 
$\mathbb{R}$--homomorphism $\iota:E \boxtimes \overline{E} \to E \boxtimes \overline{E}$ covering the map 
$X\times\overline X\ni(x,y)\to(y,x)\in X\times\overline X$, 
determined by 
\begin{equation} 
\iota(e \otimes f)= f \otimes e. 
\end{equation}
 We write 
$\cO(E\boxtimes\overline E)^\text{herm}$ for those sections $L\in\cO(E\boxtimes\overline E)$ that satisfy $\iota L(x,y)=L(y,x)$.
(3.3) shows that $K$ is in this space. Furthermore, $\iota$ leaves invariant the bundle
$E\otimes\overline E\to X$, the restriction of 
$E\boxtimes\overline E$ to the diagonal; its fixed points form a real subbundle 
$(E \otimes \overline{E})^\text{ herm} \subset E \otimes \overline{E}$. In a  
local frame $s_1,\dots,s_r$ of $E$ elements of $E_x \otimes \overline{E}_y$ are of form 
$$
t=\sum_{i,j}(a_{ij}s_i(x))\otimes s_j(y)=\sum_{i,j} s_i(x) \otimes (\overline{a_{ij}} s_j(y)), \qquad a_{ij} \in \mathbb{C}.
$$
The involution corresponds to $a_{ij} \mapsto \overline{a_{ji}}$, and elements of $(E \otimes \overline{E})^{\text{herm}}_x$ to  hermitian matrices $(a_{ij})$. It follows that $(E \otimes \overline{E})^{\textnormal{herm}}$ is indeed a subbundle, locally isomorphic to the trivial bundle over $X$ whose fibers consist of 
$r \times r$  hermitian matrices. We write $C^{\infty}(E \otimes \overline{E})$ and $C^{\infty}(E \otimes \overline{E})^{\textnormal{herm}}$ for the Fr\'echet spaces of smooth sections of $E \otimes \overline{E}$, respectively, $(E \otimes \overline{E})^{\textnormal{herm}}$. \eqref{trescuatro} shows that $k \in C^{\infty}(E \otimes \overline{E})^{\textnormal{herm}}$.  

Define linear maps 
\begin{equation} 
B:\mathcal I_E\to \cO(E\boxtimes\oE)^\text{herm}\qquad\text {and}\qquad 
\beta: \mathcal{I}_E \to C^{\infty}(E \otimes \overline{E})^{\textnormal{herm}}
\end{equation}
by sending  $\langle\langle \, , \rangle \rangle \in \mathcal{I}_E$ to its Bergman kernel $K$,
respectively Bergman section $k$, of (3.2), (3.5). 

\begin{thm} 
(a) $B$ is an isomorphism of topological vector spaces and $\beta$ is an isomorphism between the topological vector spaces 
$\mathcal{I}_E$ and $\beta(\mathcal{I}_E)\subset C^{\infty}(E \otimes \overline{E})^{\textnormal{herm}}$.

(b) $\langle \langle \, , \rangle \rangle\in\cI_E$ is positive semidefinite if and only if with $K=B(\langle \langle \, , \rangle \rangle)$ 
and arbitrary choice of finitely many $x_j\in X$ and $\xi_j\in E_{x_j}^*$ the hermitian matrix 
\begin{equation*}
\big((\xi_i\otimes\overline{\xi_j})K(x_i,x_j)\big)_{i,j}
\end{equation*}
is positive semidefinite.
\end{thm}

\begin{proof}
(a) It will suffice to prove that $B,\beta$ are vector space isomorphisms, since on a finite dimensional $\bR$--vector space
there is only one separated vector space topology.
 
We start with $B$. If $\dim\cO(E)=d$, the space $\cI_E$ is $d^2$ real dimensional. But so is the space 
$\cO(E\boxtimes\oE)^\text{herm}$, since with a basis $s_1,\dots,s_d$ of $\cO(E)$ its elements can be written 
$\sum_{i,j} a_{ij} s_i(x)\otimes s_j(y)$, where $a_{ij}=\overline{a_{ji}}$. Hence to prove isomorphism it suffices to check that
$\Ker \,B=0$. If $K=B(\langle\langle\, , \rangle\rangle)=0$, (3.2) gives  
\begin{equation} 
\langle\langle \xi \textnormal{ev}_x, \eta\, \textnormal{ev}_y \rangle\rangle=0.
\end{equation} 
Now the linear forms 
$\xi \textnormal{ev}_x \in \mathcal{O}(E)^*$ for all $x, \xi$ span $\mathcal{O}(E)^*$, because if $s \in \mathcal{O}(E)$ annihilates all of them, then $s=0$. Therefore (4.3) implies $\langle \langle \, , \rangle \rangle=0$ and $B$ is indeed isomorphic.

As to $\beta$, the point again is that it is injective. For suppose that $k=\beta (\langle \langle \, , \rangle \rangle)=0$. Thus the 
Bergman kernel $K \in \mathcal{O}(E \otimes \overline{E})$ of $\langle \langle \, , \rangle \rangle$ 
vanishes when restricted to the maximally real submanifold 
$\{(x,x): x \in X\} \subset X \times \overline{X}$.
This implies that $K=B(\langle \langle \, , \rangle \rangle) = 0$ and $\langle \langle \, , \rangle \rangle=0$ by what we have already proved.

(b) If $a_j\in\bC$, by the definition of $K$, (3.2),
\[
\langle \langle \sum_j a_j\xi_j\text{ev}_{x_j} , \sum_j a_j\xi_j\text{ev}_{x_j}\rangle \rangle=
\sum_{i,j}a_i\overline {a_j}(\xi_i\otimes\overline{\xi_j})K(x_i,x_j).
\]
As we saw in the first part of the proof, $\xi\text{ev}_x$ for all $x,\xi$ span $\cO(E)^*$, whence the claim follows.
\end{proof}

\begin{proof}[Proof of Theorem 1.2]

The map $b: \mathcal{H}_E \to C^{\infty} (E \otimes \overline{E})$ of the theorem can be expressed through $\beta$ and the duality map $\delta$, see (3.1), as $b=\beta \circ \delta | \mathcal{H}_E$. Indeed, if $\langle \, , \rangle \in \mathcal{H}_E$, a 
positive definite inner product on $\mathcal{O}(E)$, and $\langle \langle \, , \rangle \rangle= \delta (\langle \, , \rangle)$ is the dual positive definite inner product on $\mathcal{O}(E)^*$, comparing \eqref{unodos} with \eqref{trescuatro} shows that the Bergman 
section $k$ associated with $\langle \, , \rangle$ is the same as the one associated with $\langle \langle \, , \rangle \rangle$. Since 
$\delta | \mathcal{H}_E$ is a diffeomorphism on its image, which in turn is an open subset of $\mathcal{I}_E$, Theorem 4.1(a) implies 
that $b(\mathcal{H}_E)$ is indeed a submanifold of $C^{\infty}(E \otimes \overline{E})^{\textnormal{herm}}$, and $b$ is a 
diffeomorphism between $\mathcal{H}_E$ and $b(\mathcal{H}_E)$ .
\end{proof}

\section{The variational approach}
Just like in a more traditional set up, Bergman sections $k$ can be related to solutions of variational problems, and this will bring us in contact with the Fubini--Study map of the Introduction. Consider a  hermitian holomorphic vector bundle 
$(E,h) \to X$, with $X$ still compact. As a fiberwise bilinear map, $h:E \oplus \overline{E} \to \mathbb{C}$ induces a fiberwise linear map
$$
H: E \otimes \overline{E} \to \mathbb{C}, \qquad H \big(\sum_j e_j \otimes f_j \big)=\sum h(e_j, f_j).
$$
If $k=\beta (\langle \langle \, , \rangle \rangle) \in C^{\infty}(E \otimes \overline{E})^{\textnormal{herm}}$ is the Bergman section of $\langle \langle \, , \rangle \rangle \in \mathcal{I}_E$,  we define a smooth function 
\begin{equation}\label{cincouno}
\kappa=H \circ k: X \to \mathbb{R},
\end{equation}
and call it the Bergman function of $\langle \langle \, , \rangle \rangle$.

By associating $\kappa$ with $\langle \langle \, , \rangle \rangle$ we obtain a map
\begin{equation}\label{cincodos}
L:\mathcal{I}_E \to C^{\infty} (X).
\end{equation}

\begin{lem} $L$ is continuous and linear. It is also injective if $E$ is a line bundle. \end{lem}

\begin{proof}
The first statement is immediate from Theorem 4.1. As to the second, suppose $E$ is a line bundle and 
$L(\langle \langle \, , \rangle \rangle)=\kappa=0$. By \eqref{trescuatro} the Bergman section is 
$k=\sum_{1}^{p} s_j \otimes s_j-\sum_{p+1}^{p+q}s_j \otimes s_j$, with suitable linearly independent $s_j\in \mathcal{O}(E)$. Over some open $U$ fix a nonvanishing holomorphic section $e$ of $E|U$ and write $s_j=\varphi_j e,\, \varphi_j \in \mathcal{O}(U)$. On $U$ 
$$ 
0=\kappa=H\circ k= h(e,e) \Big ( \sum_{1}^{p}|\varphi_j|^2 - \sum_{p+1}^{p+q}|\varphi_j|^2 \Big).
$$
This implies $k=(e \otimes e) \big( \sum_{1}^{p}|\varphi_j|^2 - \sum_{p+1}^{p+q}|\varphi_j|^2 \big)=0$ on $U$, hence on $X$. By 
Theorem 4.1 $\langle \langle \, , \rangle \rangle=0$ follows. 
\end{proof}
Consider now an $\langle \langle \, , \rangle \rangle \in \mathcal{I}_E$, that we write as $\delta(V, \langle \, , \rangle)$, see $(3.1)$. Thus $\langle \, ,\rangle$ is a nondegenerate inner product on $V\subset \mathcal{O}(E)$, say, of signature ($p,q$). Define for  
$x \in X$ and for $l=1, 2, \dots,p$
\begin{equation}\label{cincotres}
\kappa_l(x)=\min_{W}\max_{s \in W\setminus\{0\}} \frac{h(s(x),s(x))}{\langle s, s \rangle} \geq 0,
\end{equation}
the minimum taken over $l$ dimensional subspaces $W \subset V$ on which $\langle \, , \rangle$ is positive definite; and for $l=p+1,\dots, p+q$ 
$$
\kappa_l(x)=\max_W \min_{s \in W\setminus \{0\}} \frac{h(s(x),s(x))}{\langle s, s \rangle} \leq 0,
$$
the maximum taken over $l -p$ dimensional subspaces $W \subset V$ on which $\langle \, , \rangle$ is negative definite. Since the rank of the hermitian form 
\begin{equation} 
h_x:V \times V\ni (s,t) \mapsto h(s(x),t(x))
\end{equation}
is $\le\rk  E$, when $l \leq p-\rk E$ we can choose $W$ in (5.3) contained in the kernel of $h_x$, which shows
\begin{equation}\label{cincocuatro}
\kappa_l=0 \quad\textnormal{if}\quad 1 \leq l \leq p-\rk E, \quad\textnormal{and similarly if} \quad 1 \leq l-p \leq q-\rk E.
\end{equation}
\begin{thm} With notation as above,
\begin{equation}\label{cincocinco}
\kappa=H \circ k = \sum_{1}^{p+q} \kappa_l.
\end{equation}
In particular, suppose $E$ is a line bundle. If $\langle \langle \, , \rangle \rangle$ is indefinite, 
\begin{equation}\label{cincoseis}
\kappa=\kappa_p+\kappa_{p+q};
\end{equation}
if $\langle \langle \, , \rangle \rangle$ is positive semidefinite (but not identically 0), then 
\begin{equation}\label{cincosiete}
 \kappa(x)=\max_{s \in V \setminus \{0\}} \frac{h(s(x),s(x))}{ \langle s, s \rangle}.
 \end{equation}
\end{thm}

Thus $\kappa$ as introduced in (1.1) is the special case of $\kappa$ of (5.1), when $E$ is a very ample line bundle and
$\langle\,, \rangle\in\cH_E$ (equivalently, its dual $\langle \langle \, , \rangle \rangle$ is positive definite).
  
\begin{proof} Choose a basis $s_1,\dots,s_{p+q} \in V$ that diagonalizes both $\langle \, ,\rangle$ and the positive semidefinite inner product $h_x$ of (5.4). By scaling and reordering we can arrange that
 \[ 
 \langle s_i, s_j \rangle= \begin{cases} 
 					\delta_{ij} & \textnormal{ if } i=1,\dots, p \\
					-\delta_{i j} & \textnormal{ if } i =p+1,\dots, p+q  
					\end{cases}
 \quad \textnormal{ and}\quad h(s_i(x),s_j(x))=c_i \delta_{ij},
 \]
with $0 \leq c_1 \leq\dots \leq c_p$, and $0 \leq c_{p+1} \leq \dots \leq c_{p+q}$. As 
$k=\sum_1^p s_j \otimes s_j -\sum_{p+1}^{p + q} s_j \otimes s_j$,
\begin{equation} \label{cincoocho}
\kappa(x)=\sum_1^p h(s_j(x), s_j(x))-\sum_{p+1}^{p+q} h(s_j(x), s_j (x))=\sum_1^p c_j -\sum_{p+1}^{p+q}c_j. 
\end{equation}
We will next show that $c_l=\pm\kappa_l$. If $\langle\,, \rangle$ is positive definite, this follows from Courant's minimax theorem 
\cite{Co}. In general it can be derived from Phillips's generalization of the minimax theorem \cite[Theorem 1]{P}; but since a direct 
proof following Courant's is quite straightforward, we just write it out.

Thus, suppose $W \subset V$ is $l$ dimensional and $\langle \, ,\rangle | W$ is positive definite. Since the span of $s_j$, for 
$j \geq l$, has codimension $l-1$ in $V$, it contains a nonzero vector
$s=\sum_l^{p+q} \alpha_j s_j$ that is also in $W$. Thus
\begin{equation*}
\frac{h(s(x),s(x))}{\langle s, s \rangle}= \frac{\sum_l^{p+q}c_j|\alpha_j|^2}{\sum_l^p |\alpha_j|^2-\sum_{p+1}^{p+q}|\alpha_j|^2} \geq \frac{c_l \sum_l^p|\alpha_j|^2}{\sum_l^p |\alpha_j|^2}=c_l
\end{equation*}
(note that $\sum_l^p|\alpha_j|^2 \geq \langle s, s \rangle > 0)$, and $\max_{W\setminus\{0\}} h(s(x), s(x)) / \langle s, s \rangle \geq c_l$. 
But if $W$ is the span of $s_j, j \leq l$, writing a general element of $W$ as $s= \sum_1^p \beta_j s_j \neq 0$,
$$\max_{s \in W\setminus\{0\}} \frac{h(s(x),s(x))}{\langle s, s \rangle}=\max_{\beta_j} \frac{\sum_1^{l}c_j|\beta_j|^2}{\sum_1^l |\beta_j|^2}= c_l.$$
We conclude that $\kappa_l(x)=c_l$ if $1 \leq l \leq p$. One shows similarly that $\kappa_l(x)=-c_l$ if $p+1 \leq l \leq p+q$ 
(or with  $- \langle \langle\, ,  \rangle \rangle \in \mathcal{I}_E$ one applies what has already been proved). In light of \eqref{cincoocho}, \eqref{cincocinco} follows.

If $E$ is a line bundle, because of \eqref{cincocuatro} $\kappa_l\neq 0$ only if $l=p$ or $l=p+q$. For this reason, when 
$\langle \langle \, , \rangle \rangle$ is indefinite, so that $p,q \geq 1$, \eqref{cincocinco} reduces to \eqref{cincoseis}. When 
$\langle \langle \, , \rangle \rangle \neq 0$ is positive semidefinite,  $q=0$, and the right hand side of \eqref{cincocinco} consists of 
a single term $\kappa_p$. For $l=p$,  in \eqref{cincotres} the only choice for $W$  is $V$ itself, and \eqref{cincocinco} specializes to \eqref{cincosiete}.
\end{proof}

 \begin{cor}The only semidefinite $\langle \langle \, , \rangle \rangle  \in \mathcal{I}_E$ for which $\kappa=0$ is the zero inner product.
 \end{cor}
 
\begin{proof} Say $\langle \langle \, , \rangle \rangle$ is positive semidefinite, and nonzero, so that the signature of the corresponding $\langle \, , \rangle$ on $V$ is $(p,0)$, $p\geq1$. Each $\kappa _l \geq 0, \, l=1,\dots, p,$ and $\kappa_p(x)=\max_{s \in V\setminus\{0\}} h(s(x), s(x))/\langle s, s \rangle$.
This cannot be identically zero since $V$ contains a nonzero $s$. Hence $\kappa=\sum \kappa_l \neq 0$. 
\end{proof} 
 
\section{The Fubini--Study map}

We now specialize to a very ample  hermitian holomorphic line bundle $(E,h) \to X$, and denote by $-i\omega$ the curvature form of $h$. The dual of an $\langle \, , \rangle \in \mathcal{H}_E$ is a positive definite inner product 
$\langle \langle \, , \rangle \rangle \in \mathcal{I}_E$. If we first associate with $\langle \, , \rangle$ the Bergman section 
$k=\beta (\langle \langle \, , \rangle \rangle)$, then $\kappa=H \circ k$ as in (5.1), the correspondence 
$\mathcal{H}_E \ni \langle \, , \rangle \mapsto \kappa\in C ^{\infty}(X)$ will be smooth by Lemma 5.1. Using \eqref{cincosiete} to compute $\kappa$ we see that $\kappa$ is everywhere positive. The Fubini--Study map FS: $\mathcal{H}_E \to C^{\infty}(X)$ sends $\langle \, , \rangle$ to $\log \kappa$. It is standard that FS maps into $\mathcal{H}_{\omega}=\{u \in C^{\infty}(X): \omega +i\partial \overline{\partial} u > 0 \}$, an open subset of $C^{\infty}(X)$. Indeed, with a nonvanishing holomorphic local section $e$ of $E$, writing an orthonormal basis of $(\mathcal{O}(E), \langle \, , \rangle)$, as $s_j= \varphi_j e$,   $j=1,\dots, d$, \eqref{cincouno} gives $\kappa(x)=h(e(x), e(x)) \sum |\varphi_j(x)|^2$. Hence 
\begin{equation}\label{seisuno}
\omega+i \partial \op \log \kappa= \omega + i \partial \overline{\partial} \log h (e, e)+ i\partial \overline{\partial} \log \sum | \varphi_j|^2=i \partial\overline{\partial} \log \sum |\varphi_j|^2.
\end{equation}
For fixed $x \in X$ we can choose the basis $s_j$ so that $s_1$ is orthogonal to the subspace of sections vanishing at $x$. Then $\varphi_j(x)=0$ for $j \geq 2$ (and $\varphi_1(x) \neq 0$). At $x$ we obtain
\begin{equation}\label{seisdos}
i\partial \overline{\partial} \log \sum | \varphi_j|^2= i \sum_{j\ge 2} \partial \varphi_j \wedge\overline{\partial \varphi_j} / |\varphi_1|^2.
\end{equation}
This latter $(1,1)$ form is positive: as $E$ is very ample, for any nonzero tangent vector $v \in T^{10}_x X$ there is an 
$s\in \mathcal{O}(E)$ that vanishes at $x$, to first order in the direction $v$, i.e., if $s=\varphi e$ then $\varphi(x)=0$, 
$\partial \varphi (v) \neq 0$; and this $s$, resp. $\varphi$, must be in the span of $s_j$, resp. $\varphi_j$, $j \geq 2$. 

This computation proves something more general. Let us say that a subspace $V \subset \mathcal{O}(E)$ is rather ample if for every $x \in X$ and nonzero $v \in T_x^{10} X$ there are sections $s, t \in V$ such that $s(x) \neq 0$ but $t$ vanishes at $x$, to first order in the direction $v$. Geometrically this means that $X \ni x \mapsto V \cap \Ker\,\text{ev}_x \in \mathbb{P}(V)$ defines an immersion into the projective space of hyperplanes in $V$. 

\begin{lem} Consider a subspace $V \subset \mathcal{O}(E)$, and a positive definite inner product $\langle \, , \rangle$ on it. The Bergman function $\kappa$  associated with $\langle \langle \, , \rangle \rangle=\delta (V, \langle \, ,  \rangle)$, see \eqref{cincouno}, satisfies $\log \kappa \in \mathcal{H}_{\omega}$ if and only if $V$ is rather ample.
\end{lem}
\begin{proof}
If $V$ is rather ample, then $\kappa >0$, and the computation \eqref{seisuno}, \eqref{seisdos} above, with $\mathcal{O}(E)$ replaced by $V$, gives $\log \kappa \in \mathcal{H}_{\omega}$. Conversely, if $s_j$ form an orthonormal basis of $(V, \langle \, , \rangle)$ and $\kappa(x)= \sum h(s_j(x), s_j(x))>0$ for all $x \in X$,
then for every $x \in X$ there is an $s=s_j \in V$ such that $s_j(x) \neq 0$. Suppose furthermore $\log \kappa \in \mathcal{H}_{\omega}$. Given $x \in X$, if the orthonormal basis is chosen so that $s_j(x)=0$ for $j \geq 2,$ then \eqref{seisuno}, \eqref{seisdos} show that for every nonzero $v \in T_x^{10} X$ one of the $s_j$, $j \geq 2$, will vanish to first order in the direction of $v$.
 \end{proof}
Let us return to the Fubini--Study map, and prove that for a very ample line bundle $(E,h) \to X$ it is a diffeomorphism between $\mathcal{H}_E$ and $\text{FS}(\mathcal{H}_E)$, a submanifold of $\mathcal{H}_{\omega} \subset C^{\infty}(X):$ 

\begin{proof}[Proof of Theorem 1.1] 
In Lemma 5.1 we associated with any $\langle \langle \, , \rangle \rangle \in \mathcal{I}_E$ its Bergman function 
$L(\langle \langle \, , \rangle \rangle) = \kappa \in C^{\infty}(X)$, and saw that $L$ was continuous, linear , and injective, hence an isomorphism of topological vector spaces $\mathcal{I}_E \to L(\mathcal{I}_E)$. Now FS can be obtained from $L$ by precomposing it with $\delta | \mathcal{H}_E$ (see (3.1)) and by postcomposing with $\log$. As $\delta | \mathcal{H}_E$ is a diffeomorphism to its image, an open subset of $\mathcal{I}_E$, the claim follows if we take into account Lemma 6.1.
\end{proof}

It is noteworthy, though, that in general  FS: $\mathcal{H}_E \to \mathcal{H}_{\omega}$ is not proper, and FS$(\mathcal{H}_E) \subset \mathcal{H}_{\omega}$ is not a closed submanifold.

\begin{example} If. $X=\mathbb{P}_1, E \to X$ is the $d$'th power of the hyperplane section bundle, $d \geq 4$, and $h$ is any hermitian metric on $E$, then FS$(\mathcal{H}_E)$ is not a relatively closed subset of $\cH_{\omega}$.
\end{example}

\begin{proof}  In the standard trivialization of $E$ over $\mathbb{C}=\mathbb{P}_1 \setminus \{\infty\}$, sections of $E$ 
correspond to polynomials of degree $\leq d$, and  $h$ can be expressed with a suitable $a\in C^\infty(\bC)$ as
$$
h(\zeta_1, \zeta_2)= a(z)\zeta_1 \overline{\zeta_2}, \qquad z\in\bC,\quad\zeta_1, \zeta_2 \in E_z \approx \mathbb{C}. 
$$
With $t \in (0,\infty)$ define $ \langle \, , \rangle_t \in \mathcal{H}_E$ so that $1, z,tz^2, z^3,\dots, z^d$ form an orthonormal basis. 
The corresponding Bergman section is $k_t(z)=\sum ' z^j \otimes z^j+ t z^2 \otimes z^2$, where $\sum'$ refers to summation over 
$j$ omitting  $j=2$; and 
$\kappa_t(z)=a(z)(\sum'|z|^{2j}+ t^2 |z|^4)$. As $t \to 0$, the limit of FS$(\langle \, ,\rangle_t)=\log \kappa_t$ is 
$$
\log \kappa_0(z)=\log a(z)+\log {\sum}' |z|^{2j}.
$$ 
Since $z^j$ for $j\neq 2$ span a rather ample subspace $V \subset \mathcal{O}(E)$,
Lemma $6.1$ implies $\log \kappa_0 \in \mathcal{H}_{\omega}$. It is also in the closure of FS$(\cH_E)$.

However, $\log\kappa_0 \notin \textnormal{FS}(\mathcal{H}_E)$. Indeed, on the one hand, $L:\mathcal{I}_E \to C^{\infty}(X)$ is 
injective by Lemma 5.1. On the other, if $V \subset \mathcal{O}(E)$ is the span of $z^j, j \neq 2$, an inner product is defined on it
by  $\langle z^i, z^j \rangle= \delta_{ij}$, and  $\delta(V, \langle \, , \rangle)=\langle \langle \, , \rangle \rangle$, then
$\kappa_0=L(\langle \langle \, , \rangle \rangle)$. Since this 
$\langle \langle \, , \rangle \rangle$ is not positive definite, it is not in $\delta(\mathcal{H}_E)$. Hence 
$\kappa_0 \notin L(\delta(\mathcal{H}_E))$ and $\log \kappa_0 \notin \text{FS} (\mathcal{H}_E)$.
\end{proof}

Nonetheless, it is not hard to describe, for a general very ample hermitian line bundle $(E,h) \to X$, the closure of 
FS$(\mathcal{H}_E) \subset \mathcal{H}_{\omega}$. This can be done most conveniently through the `dual' Fubini--Study map 
$\Phi$, defined on certain positive semidefinite $\langle \langle \, , \rangle \rangle \in \mathcal{I}_E$.

\begin{defn} Let $\mathcal{A}_E \subset \mathcal{I}_E$ consist of positive semidefinite inner products 
$\langle \langle \, , \rangle \rangle=\delta(V, \langle \, , \rangle)$ on $\mathcal{O}(E)^*$ for which $V$ is rather ample. If 
$\langle \langle \, , \rangle \rangle \in \mathcal{A}_E$, define 
$\Phi(\langle \langle \, , \rangle \rangle)= \log L(\langle \langle \, , \rangle \rangle)$.
\end{defn}
Thus $\Phi(\langle \langle \, , \rangle \rangle)=\log \kappa$, if $\kappa \in C^{\infty}(X)$ is the Bergman function of 
$\langle \langle \, , \rangle \rangle$. By Lemma 6.1, $\kappa$ is positive and $\log \kappa \in \mathcal{H}_{\omega}$. We call 
$\Phi : \mathcal{A}_E \to \mathcal{H}_{\omega}$ the dual Fubini--Study map, for
\begin{equation}\label{seistres}
\textnormal{FS}=\Phi \circ \delta | \mathcal{H}_E.
\end{equation}
\begin{thm}The dual Fubini--Study map $\Phi$ is proper and injective, and the closure of FS$(\mathcal{H}_E)$ in $\mathcal{H}_{\omega}$ is $\Phi(\mathcal{A}_E)$ .
\end{thm}
\begin{proof} Since by Lemma 5.1 $L$ is injective, so is $\Phi$. To prove it is proper, take 
$\langle \langle \, , \rangle \rangle_j \in \mathcal{A}_E, j \in \mathbb{N}$, and suppose 
$\lim_{j \to \infty} \Phi (\langle \langle \, , \rangle \rangle_j)=\varphi \in \mathcal{H}_{\omega}$. Then 
$\lim_{j \to \infty} L(\langle \langle \, , \rangle \rangle_j)=e^{\varphi} $ in $C^{\infty}(X)$. But $L(\mathcal{I}_E) \subset C^{\infty}(X)$ is a finite dimensional subspace, hence closed; and $L^{-1}: L(\mathcal{I}_E) \to \mathcal{I}_E$ is continuous. Therefore $e^{\varphi} \in L(\mathcal{I}_E)$, and $\langle \langle \, , \rangle \rangle_j$ converges to 
$\langle \langle \, , \rangle \rangle= L^{-1}(e^{\varphi})\in \mathcal{I}_E$, necessarily positive semidefinite. Thus 
$\kappa=e^{\varphi}= L(\langle \langle \, , \rangle \rangle)$, and $\log \kappa \in \mathcal{H}_{\omega}$.
Suppose $\langle \langle \, , \rangle \rangle= \delta(V, \langle \, , \rangle)$. Lemma 6.1 gives that $V$ is rather ample, whence 
$\langle \langle \, , \rangle \rangle \in \mathcal{A}_E$ and $\varphi \in \Phi(\mathcal{A}_E)$. This shows that 
$\Phi (\mathcal{A}_{E})$ is closed in $\mathcal{H}_{\omega}$ and $\Phi$ is proper.

By \eqref{seistres} the closure of FS$(\mathcal{H}_E)$ is contained in $\Phi (\mathcal{A}_E)$.

Conversely, suppose $\langle \langle \, , \rangle \rangle \in \mathcal{A}_E$. If $\langle \langle \, , \rangle \rangle_0 \in \mathcal{I}_E$ is positive definite, then $\langle \langle \, , \rangle \rangle + t \langle \langle \, , \rangle \rangle_0 \in \mathcal{I}_E$ is positive definite for all $t >0$, and of form $\delta (\langle \, , \rangle_t)$ with some $\langle \, , \rangle_t \in \mathcal{H}_E$. It follows that
$$\Phi (\langle \langle \, , \rangle \rangle)= \lim_{ t \to 0} \Phi (\langle \langle \, , \rangle \rangle + t \langle \langle \, , \rangle \rangle_0)=\lim_{t \to 0} \text{FS}(\langle \, , \rangle_t)$$
is in the closure of FS$(\mathcal{H}_E)$, and this closure is $\Phi(\mathcal{A}_E)$.
\end{proof}

Of course, in general $\Phi(\mathcal{A}_E) \subset \mathcal{H}_{\omega}$, even if closed, is not a submanifold. Still, FS$(\mathcal{H}_E)$ can be continued analytically to a closed submanifold of $\mathcal{H}_{\omega}$. For this one considers $\mathcal{P}_E \subset \mathcal{I}_E$ consisting of not necessarily semidefinite inner products $\langle \langle \, , \rangle \rangle$ whose Bergman function $\kappa =L(\langle \langle \, , \rangle \rangle)$ is everywhere positive and $\log \kappa \in \mathcal{H}_{\omega}$. This is the open subset of $\mathcal{I}_E$ that $\log L$ maps in $\mathcal{H}_{\omega}$; and $\log L(\mathcal{P}_E)$ is a closed submanifold (connected and analytic) of $\mathcal{H}_{\omega}$, containing FS$(\mathcal{H}_{\omega})$ as an open subset.

\end{document}